\documentclass[12pt]{article}
\usepackage{amssymb,amsmath,graphicx,acronym,amsfonts,float,color,epstopdf}
\usepackage{algorithmicx,algorithm}
\usepackage{graphicx}
\usepackage{subfigure}
\usepackage{caption}
\usepackage{graphicx,subfig}
\usepackage[numbers,sort&compress]{natbib}
\usepackage{amsmath}
\allowdisplaybreaks[4]

\newtheorem{theorem}{Theorem}[section]

\newtheorem{proposition}{Proposition}[section]
\newtheorem{remark}{Remark}[section]
\newtheorem{lemma}{Lemma}[section]
\newtheorem{example}{Example}[section]
\newtheorem{definition}{Definition}

\def\bt{\begin{theorem}}
\def\et{\end{theorem}\bigskip}
\def\bt{\begin{Remark}}
\def\et{\end{Remark}\bigskip}
\def\bl{\begin{Lemma}}
\def\el{\end{Lemma}\bigskip}
\def\ep{\end{Proposition}\bigskip}
\def\bp{\begin{Proposition}}
\def\bd{\begin{definition}}
\def\ed{\end{definition}}
\definecolor{red8}{rgb}{0.8,0.0,0.8}
\definecolor{red5}{rgb}{0.5,0.0,0.5}
\definecolor{red6}{rgb}{0.6,0.0,0.6}
\definecolor{red7}{rgb}{0.7,0.0,0.7}
\definecolor{red9}{rgb}{0.9,0.0,0.9}
\definecolor{green2}{rgb}{0.0,0.2,0.0}
\definecolor{green3}{rgb}{0.0,0.3,0.0}
\definecolor{green4}{rgb}{0.0,0.4,0.0}
\definecolor{blue4}{rgb}{0.4,0.0,1.0}
\definecolor{yellow4}{rgb}{0.9,0.8,0.5}
\definecolor{blue5}{rgb}{0.0,0,0.5}
\definecolor{blue6}{rgb}{0.0,0.5,0.7}
\definecolor{hong1}{rgb}{0.8,0.8,0.8}

\newcommand{\I}{{\cal I}}

\input epsf
\begin{document}

\title{\bf The Sparse Solution to $\mathcal{KS}$-Tensor Complementarity Problems}

\author{Jingjing Sun\footnote{School of Mathematics and Statistics,
Qingdao University, Qingdao,  266071, China.}\qquad Shouqiang Du\footnote{School of Mathematics and Statistics,
Qingdao University, Qingdao,  266071, China.}\qquad Yuanyuan Chen\footnote{School of Mathematics and Statistics,
Qingdao University, Qingdao,  266071, China.}\qquad Yimin Wei\footnote{Corresponding author. School of Mathematical Science and Shanghai Key Laboratory of Contemporary Applied Mathematics,
Fudan University, Shanghai,  610101, This author is supported by the National Natural Science Foundation of China under grant 11771099 and  Innovation Program of Shanghai Municipal Education Commission, China. E-mail: ymwei@fudan.edu.cn and yimin.wei@gmail.com}}

\date{}
\maketitle

\begin{abstract} In view of the $\mathcal{KS}$-tensor complementarity problem, the sparse solution of this problem is studied. Due to the nonconvexity and noncontinuity of the $l_0$-norm, it is a NP hard problem to find the sparse solution of the $\mathcal{KS}$-tensor complementarity problem. In order to solve this problem, we transform it into a polynomial programming problem with constraints. Then we use the sequential quadratic programming (SQP) algorithm to solve this transformed problem. Numerical results show that the SQP algorithm can find the sparse solutions of the $\mathcal{KS}$-tensor complementarity problem effectively.



\end{abstract}

\newpage

\section{Introduction}

The complementarity problem is a classical and important research topic in the optimization field. The study on applications of the complementarity problem is an important issue in the engineering and the traffic balance problem \cite{CPS,GP,PSS,FP,HXQ,LHH,KTX}. Tensor complementarity problem (TCP) was firstly introduced in \cite{SQ15}. It is a subset of nonlinear complementarity problem (NCP) \cite{FP,HXQ}. TCP is widely used in the fields of the $n$-person non-cooperative game theory, the hypergraph clustering and related problems in recent years \cite{{HQ},{CZ},{HQIII},{WDm}}.

Recently, many kinds of TCP with different structured tensors have been studied, such as the $\mathcal{Z}$-tensor \cite{GLQX}, the positive-definite tensor \cite{CQW}, the $\mathcal{P}$-tensor, the strong $\mathcal{P}$-tensor \cite{BHW} and the copositive tensor \cite{YLH}. More theoretical results of the TCP can be found in \cite{SQ15,GLQX,SY,DDW,SQ17,DLQ,BP,CQS,ZW,LLV,HV,HQI,HQII} and the references therein. Just recently, a new class of tensor called $\mathcal{KS}$-tensor was presented in \cite{WMCW}. It is a subclass of $\mathcal{P}$-tensor and a generalization of $\mathcal{H^{+}}$-tensor. We naturally have the following questions. Can we investigate the properties of solution set for TCP with $\mathcal{KS}$-tensor? Can the solution set of $\mathcal{KS}$-tensor complementarity problems also have nice properties? We will answer the given questions in the following of this paper. Now, we firstly propose the $\mathcal{KS}$-tensor complementarity problem ($\mathcal{KS}$-TCP), which is to find a vector $x\in\mathbb{R}^{n}$ satisfying
\begin{equation}\label{kstcp}
\mathcal{A}x^{m-1}-q\geq0,~x\geq0,~x^{\top}(\mathcal{A}x^{m-1}-q)=0,
\end{equation}
where $q\in{\mathbb{R}^{n}}$ and $\mathcal{A}$ is a $\mathcal{KS}$-tensor.

On the other hand, there are many research results about the sparse solutions of the optimization problem, such as \cite{LQKX,SZX,LK,LQX,LLX,XLX}.
As for the TCP, we know that the research work about sparse solution of the TCP was only reported in \cite{LQX,XLX}. We will consider the sparse solution of $\mathcal{KS}$-TCP, which is also our motivation for this paper.

The main contributions of this paper are listed as follows.
\begin{enumerate}
  \item We study the properties of $\mathcal{KS}$-tensor, and introduce the differences and the connections between $\mathcal{KS}$-tensor and other structured tensors. We prove that the solution set of $\mathcal{KS}$-TCP is nonempty and compact.
  \item We transform the sparse solution problem of $\mathcal{KS}$-TCP into a polynomial programming problem with constraints.
  \item The SQP algorithm is proposed to find the sparse solution of $\mathcal{KS}$-TCP.
  \end{enumerate}

This paper is organized as follows. In Section 2, we recall some basic definitions and some properties about tensors and $Z$-function. In Section 3, we propose the difference properties and the connection properties between $\mathcal{KS}$-tensor and some kinds of structured tensors, and present the properties of the sparse solutions of $\mathcal{KS}$-TCP. In Section 4, we apply the SQP algorithm to solve the transformed optimization problem, and then obtain a sparse solution to $\mathcal{KS}$-TCP. In Section 5, numerical results of the sparse solution to $\mathcal{KS}$-TCP are given. Finally, We give our conclusions in Section 6.
\section{Preliminaries}
In this section, we introduce some  definitions and useful properties, which will be used in the following of this paper.

Throughout this paper, let $\mathbb{R}^n:=\{(x_1,x_2,\ldots,x_n)^{\top}: x_i\in \mathbb{R}, i\in[n]=1,2,\ldots,n\}$, where
$\mathbb{R}$ is the set of real numbers.
We use $\mathbb{R}^{n}_{+}$ and $\mathbb{R}^{n}_{++}$ to denote the set of  $n$-dimensional non-negative  and positive  vectors, respectively.
The definitions and computations of tensor have been introduced in \cite{WDm,QCCm,QLm}. Let $\mathcal{A}=(a_{i_1\cdots i_m})\in \mathbb{R}^{[m,n]}$ and $x\in \mathbb{R}^n$, where $\mathbb{R}^{[m,n]}$ is the set of all real tensors of order $m$ and dimension $n$. Then $\mathcal{A}x^{m-1}\in\mathbb{R}^n$ is a vector defined by
$$
(\mathcal{A}x^{m-1})_{i}:=\sum_{i_2,i_3,\cdots i_m=1}^na_{ii_2\cdots i_m}x_{i_2}x_{i_3}\cdots x_{i_m}, \forall i\in\lbrack n\rbrack.
$$

$\mathcal{A}\in \mathbb{R}^{[m,n]}$ is called {\it non-negative tensor} if each entry is non-negative \cite{XQW}. An $m$-order $n$-dimensional tensor $\mathcal{I}=(a_{i_1\ldots i_m})\in \mathbb{R}^{[m,n]}$ is called an {\it identity tensor} if\\
\begin{center}
$a_{i_1\cdots i_m}=\left\{\begin{array}{lc}1,&if\;i_1=\cdots=i_m,\\0,&otherwise.\end{array}\right.$
\end{center}

 $\mathcal{A}=(a_{i_1i_2\cdots i_m})\in \mathbb{R}^{[m,n]}$ is a {\it partially symmetric tensor} \cite{H17}, if each entry of $\bar{\mathcal{A}}=(\bar{a}_{i_1i_2\cdots i_m})$ satisfies

\begin{equation}\label{a-}
  \bar{a}_{i_1i_2\cdots i_m}=\frac1{(m-1)!}\sum_{\pi(i_2\cdots i_m)}a_{i_1\pi(i_2\cdots i_m)},\nonumber
\end{equation}
where the summation is over all the permutations $\pi(i_2\cdots i_m)$.

Moreover, for any $x\in \mathbb{R}^n$, the $l_0$-norm $\|x\|_0$ denotes the number of nonzero elements of $x\in \mathbb{R}^n$. $l_2$-norm $\|x\|_2$ is denoted by
$$
\|x\|_2=\sqrt{\sum_{i=1}^nx_i^2}.
$$
$x^{\top}$ represents the transpose of $x$. Given two  vectors $x,y\in \mathbb{R}^n$, $x^{\top} y= \left\langle x,y\right\rangle$ represents the inner product of $x$ and $y$.

Next, we recall the concepts of $\mathcal{P}$-tensor and $\mathcal{P}_0$-tensor \cite{SQ15}.

\begin{definition}\label{defp}
 Let $\mathcal{A}=(a_{i_1\cdots i_m})\in \mathbb{R}^{[m,n]}$, then $\mathcal{A}$ is defined as follows:
\begin{description}
  \item[(a)] a $\mathcal{P}$-tensor, if and only if for each  $x\in \mathbb{R}^n$, there exists an index $i\in[n]$ such that $x_i\neq 0$ and $x_i({\mathcal{A}}x^{m-1})_i\ge 0$.
  \item[(b)] a $\mathcal{P}_0$-tensor, if and only if for each  $x\in \mathbb{R}^n$, there exists an index $i\in[n]$ such that $x_i\neq 0$ and $x_i({\mathcal{A}}x^{m-1})_i>0$.
\end{description}
\end{definition}

From \cite{BHW}, we know that a tensor $\mathcal{A}\in\mathbb{R}^{[m,n]}$ is a nonsingular $\mathcal{P}$-tensor if $F(x):=\mathcal{A}x^{m-1}$ is a $P$-function on $\mathbb{R}^n$, i.e., for every $x,y\in \mathbb{R}^n$, $\underset{i\in\left[n\right]}{max}(x_i-y_i){(\mathcal{A}x^{m-1}-\mathcal{A}y^{m-1})}_i>0$.

\begin{definition}
Let $A\in \mathbb{R}^{n \times n}$. $A$ is called a $Z$-matrix, if all its off-diagonal entries are non-positive.
\end{definition}

\begin{definition}{\rm \cite{I}}
A mapping $F:   \mathbb{R}^n\rightarrow \mathbb{R}^n$ is said to be a $Z$-function if for every $x,y,z\in \mathbb{R}_+^n$ such that the inner product $\left\langle x,y-z\right\rangle=0$,  $\left\langle x,F(y)-F(z)\right\rangle\leq0$ holds.
\end{definition}

\begin{proposition}\label{lem1}{\rm \cite{LQX}}
A G$\hat{a}$teaux continuous differentiable function $F:\mathbb{R}^{n}\rightarrow \mathbb{R}^{n}$ is a $Z$-function if and only if for $\forall x\in \mathbb{R}^{n}_{+}$, $\nabla F(x)$ is a $Z$-matrix.
\end{proposition}

\begin{definition}{\rm \cite{ZQZ}}
Let $\mathcal{A}=(a_{i_1\ldots i_m})\in \mathbb{R}^{[m,n]}$. $\mathcal{A}$ is called a $\mathcal{Z}$-tensor, if all its off-diagonal entries are non-positive.
\end{definition}

\begin{definition}{\rm \cite{ZQZ}}
Let $\mathcal{A}=(a_{i_1\ldots i_m})\in \mathbb{R}^{[m,n]}$. $\mathcal{A}$ is called an $\mathcal{M}$-tensor, if it can be written as $\mathcal{A}=s\mathcal{I}_m-\mathcal{B}$, where $\mathcal{I}_m$ is the $m$-order $n$-dimensional identity tensor, $\mathcal{B}$ is a non-negative tensor, and $s\geq\rho(\mathcal B)$. Furthermore, $\mathcal{A}$ is called a nonsingular (strong) $\mathcal{M}$-tensor if $s>\rho(\mathcal B)$. Here, $\rho(\mathcal B)$ is the spectral radius of $\mathcal{B}$.
\end{definition}

Some important properties of nonsingular $\mathcal{M}$-tensors are shown below, which are necessary in the subsequential analysis.
\begin{proposition}\label{thM}{\rm \cite{XQW}}
Let $\mathcal{A}\in \mathbb{R}^{[m,n]}$ be a $\mathcal{Z}$-tensor, then the following propositions are equivalent:
\begin{enumerate}[\rm (i)]
  \item $\mathcal{A}$ is a nonsingular $\mathcal{M}$-tensor;
  \item There exists an $x\in \mathbb{R}_{++}^n$ satisfying $\mathcal{A}x^{m-1} \in \mathbb{R}_{++}^n$;
  \item $\mathcal{A}$ is a $\mathcal{P}$-tensor.
\end{enumerate}
\end{proposition}

Next, we will introduce a new class of tensors called $\mathcal{KS}$-tensors, which is a subset of nonsingular $\mathcal{P}$-tensors and generalization of $\mathcal{H}^{+}$-tensors \cite{WMCW}.

\begin{definition}\label{defks}
Let $\mathcal{A}=(a_{i_1\ldots i_m})\in \mathbb{R}^{[m,n]}$ be a $\mathcal{P}$-tensor. $\mathcal{W}=(w_{i_1i_2\dots i_m})$ is a tensor with entries
\begin{equation}
  w_{i_1i_2\cdots i_m}=\left\{\begin{array}{lc}a_{i_1i_2\cdots i_m}&if\;i_1=i_2=\cdots=i_m,\\0&if\;\delta_{i_1i_2\cdots i_m}=0\;and\;a_{i_1i_2\cdots i_m}>0,\;\\a_{i_1i_2\cdots i_m}&if\;\delta_{i_1i_2\cdots i_m}=0\;and\;a_{i_1i_2\cdots i_m}\leq0.\end{array}\right.\label{1}
\end{equation}
If $\mathcal{W}$ is a non-singular (strong) $\mathcal{M}$-tensor, then $\mathcal{A}$ is called a $\mathcal{KS}$-tensor.
\end{definition}

By the above definition, we can see that any $\mathcal{KS}$-tensor $\mathcal{A}$ can be written as $\mathcal{A}=\mathcal{W}+\mathcal{N}$ where $\mathcal{W}$ is given as in Definition \ref{defks} and $\mathcal{N}$ is a non-negative tensor. And we also have the following proposition.
\begin{proposition}\label{lem2.1}
Let $\mathcal{A}\in\mathbb{R}^{[m,n]}$ be a $\mathcal{KS}$-tensor. Then $\mathcal{W}$ is the $\mathcal{Z}$ tensor with positive diagonal elements, and $\mathcal{N}$ is the non-negative tensor with zero diagonal elements.
\end{proposition}
\begin{proof}
Since $\mathcal{A}$ is a $\mathcal{P}$-tensor, all diagonal elements of it are positive. As we know from \eqref{1}, all diagonal elements of $\mathcal{W}$ are positive. Thus, the off diagonal elements of $\mathcal{W}$ are non-positive, we know that $\mathcal{W}$ is the $\mathcal{Z}$ tensor with positive diagonal elements.

From the definition of $\mathcal{W}$, $\mathcal{N}$ is the tensor with entries
\begin{center}
$n_{i_1i_2\cdots i_m}=\left\{\begin{array}{lc}0&if\;i_1=i_2=\cdots=i_m,\\a_{i_1i_2\cdots i_m}&if\;\delta_{i_1i_2\cdots i_m}=0\;and\;a_{i_1i_2\cdots i_m}>0,\;\\0&if\;\delta_{i_1i_2\cdots i_m}=0\;and\;a_{i_1i_2\cdots i_m}\leq0.\end{array}\right.$
\end{center}
Thus $\mathcal{N}$ is the non-negative tensor with zero diagonal elements.
\end{proof}

\begin{remark}
According to Definition \ref{defks}, it is obvious that the $\mathcal{KS}$-tensor is not necessarily a $\mathcal{Z}$-tensor.
\end{remark}
Next, we give an example to show that the $\mathcal{KS}$-tensor is not necessarily a $\mathcal{Z}$-tensor.
\begin{example}
Let $\mathcal{A}$ be a 3-order 2-dimensional tensor with entries
\begin{center}
$\mathcal{A}(i,j,1)=\begin{pmatrix}1&0\\1&0\end{pmatrix}$ and $\mathcal{A}(i,j,2)=\begin{pmatrix}0&-1\\0&1\end{pmatrix}$.
\end{center}
For any non-zero vector $x=(x_1,x_2)^{\top}$, we have
\begin{center}
$\mathcal{A}x^{m-1}=\begin{pmatrix}x_1^2-x_2^2\\x_1^2+x_2^2\end{pmatrix}$.
\end{center}
Thus, $x_1(\mathcal{A}x^2)_1=x_1(x_1^2-x_2^2)$ and $x_2(\mathcal{A}x^2)_2=x_2(x_1^2+x_2^2)$. We can see that $\mathcal{A}$ is a $\mathcal{P}$-tensor. Then
$$
\begin{cases}
\mathcal{W}(i,j,1)=\begin{pmatrix}1&0\\0&0\end{pmatrix},\quad \mathcal{W}(i,j,2)=\begin{pmatrix}0&-1\\0&1\end{pmatrix},\\
\mathcal{N}(i,j,1)=\begin{pmatrix}0&0\\1&0\end{pmatrix}, \quad \mathcal{N}(i,j,2)=\begin{pmatrix}0&0\\0&0\end{pmatrix}.
\end{cases}
$$
We can find that $\mathcal{W}$ is a non-singular (strong) $\mathcal{M}$-tensor. Thus $\mathcal{A}$ is a $\mathcal{KS}$-tensor.
However, $\mathcal{A}(2,1,1)=1>0$ indicates that $\mathcal{A}$ is not a $\mathcal{Z}$-tensor.
\end{example}

\begin{remark}
A $\mathcal{Z}$-tensor is not necessarily a $\mathcal{KS}$-tensor.
\end{remark}
Similarly, we also give an example to show that a $\mathcal{Z}$-tensor is not necessarily a $\mathcal{KS}$-tensor.
\begin{example}
Let $\mathcal{A}$ be a 3-order 2-dimensional tensor with entries
\begin{center}
$\mathcal{A}(i,j,1)=\begin{pmatrix}1&-1\\0&-1\end{pmatrix}$ and $\mathcal{A}(i,j,2)=\begin{pmatrix}-2&0\\0&-1\end{pmatrix}$.
\end{center}
Obviously, $\mathcal{A}$ is a $\mathcal{Z}$-tensor.\\
For any non-zero vector $x=(x_1,x_2)^{\top}$, we have
\begin{center}
$\mathcal{A}x^{m-1}=\begin{pmatrix}x_1^2-3x_1x_2\\-x_2^2-x_1x_2\end{pmatrix}$.
\end{center}
Thus, $x_1(\mathcal{A}x^2)_1=x_1(x_1^2-3x_1x_2)$ and $x_2(\mathcal{A}x^{m-1})_2=x_2(-x_2^2-x_1x_2)$. We can see that $\mathcal{A}$ is not a $\mathcal{P}$-tensor. Then $\mathcal{A}$ is not a $\mathcal{KS}$-tensor.
\end{example}

Theorem 3.1 in \cite{WMCW} shows that a $\mathcal{Z}$-tensor is a $\mathcal{KS}$-tensor when $\mathcal{A}$ satisfies certain conditions, as shown in the following proposition.
\begin{proposition}\label{thm0}
Let $\mathcal{A}=(a_{i_1i_2\cdots i_m})\in \mathbb{R}^{[m,n]}$ be a $\mathcal{Z}$-tensor, then $\mathcal{A}$ is a $\mathcal{KS}$-tensor if and only if $\mathcal{A}$ is a non-singular (strong) $\mathcal{M}$-tensor.
\end{proposition}

In the following, we mainly consider the problem \eqref{kstcp} which the $\mathcal{KS}$-tensor satisfies the condition:
\begin{equation}
w_{ii_2\cdots i_m}+w_{i_2i\cdots i_m}+\cdots+w_{i_2\cdots i_mi}\leq -n_{ii_2\cdots i_m}-n_{i_2i\cdots i_m}-\cdots-n_{i_2\cdots i_mi}, i\neq i_m.\label{2}
\end{equation}

Next, we give an example to show that the $\mathcal{KS}$-tensor is different from the $\mathcal{Z}$-tensor  if condition \eqref{2} is satisfied.
\begin{example}\label{teshuKSnoZ}
Let $\mathcal{A}=(a_{i_1i_2i_3i_4})\in \mathbb{R}^{[4,2]}$ with entries $a_{1111}=a_{2222}=a_{1212}=1$, $a_{1221}=-1$, $a_{2112}=-0.5$ and other $a_{i_1i_2i_3i_4}=0$.
\end{example}
For any vector $x=(x_1,x_2)^{\top}$, we have
\begin{center}
$\mathcal{A}x^3=\begin{pmatrix}x_1^3\\x_2^3-0.5x_1^2x_2\end{pmatrix}$.
\end{center}
Thus, $x_1(\mathcal{A}x^3)_1=x_1^4$ and $x_2(\mathcal{A}x^3)_2=x_2^2(x_2^2-0.5x_1^2)$. By Definition \ref{defp}, $\mathcal{A}$ is a $\mathcal{P}$-tensor. It can be seen from formula \eqref{1}, $w_{1111}=w_{2222}=1$, $w_{1212}=0$, $w_{1221}=-1$, $w_{2112}=-0.5$. Obviously, $\mathcal{W}$ is a $\mathcal{Z}$-tensor. From
\begin{center}
$\mathcal{W}x^3=\begin{pmatrix}x_1(x_1^2-x_2^2)\\x_2(x_2^2-0.5x_1^2)\end{pmatrix}$,
\end{center}
there exists $x=(1.4,1.3)^{\top}>0$ such that $\mathcal{W}x^3>0$. Thus  $\mathcal{W}$ is a nonsingular $\mathcal{M}$-tensor, and then $\mathcal{A}$ is a $\mathcal{KS}$-tensor. Next, we will verify that $\mathcal{A}$ satisfies condition \eqref{2} through the discussion of different situations:
\begin{enumerate}
  \item $i=1, i_4=2$
  \begin{itemize}
    \item $i_2=1, i_3=2$: $w_{1122}+w_{1122}+w_{1212}+w_{1221}\leq-n_{1122}-n_{1122}-n_{1212}-n_{1221}$
    \item $i_2=2, i_3=1$: $w_{1212}+w_{2112}+w_{2112}+w_{2121}\leq-n_{1212}-n_{2112}-n_{2112}-n_{2121}$
    \item $i_2=1, i_3=1$: $w_{1112}+w_{1112}+w_{1112}+w_{1121}\leq-n_{1112}-n_{1112}-n_{1112}-n_{1121}$
    \item $i_2=2, i_3=2$: $w_{1222}+w_{2122}+w_{2212}+w_{2221}\leq-n_{1222}-n_{2122}-n_{2212}-n_{2221}$
  \end{itemize}
  \item $i=2, i_4=1$
  \begin{itemize}
    \item $i_2=1, i_3=2$: $w_{2121}+w_{1221}+w_{1221}+w_{1212}\leq-n_{2121}-n_{1221}-n_{1221}-n_{1212}$
    \item $i_2=2, i_3=1$: $w_{2211}+w_{2211}+w_{2121}+w_{2112}\leq-n_{2211}-n_{2211}-n_{2121}-n_{2112}$
    \item $i_2=1, i_3=1$: $w_{2111}+w_{1211}+w_{1121}+w_{1112}\leq-n_{2111}-n_{1211}-n_{1121}-n_{1112}$
    \item $i_2=2, i_3=2$: $w_{2221}+w_{2221}+w_{2221}+w_{2212}\leq-n_{2221}-n_{2221}-n_{2221}-n_{2212}$
  \end{itemize}
\end{enumerate}
It can be seen from the above analysis, $\mathcal{A}$ is a $\mathcal{KS}$-tensor instead of  a $\mathcal{Z}$-tensor.

\section{The sparse solution to $\mathcal{KS}$-tensor complementarity problem}

In this section, we consider the sparsest solution to $\mathcal{KS}$-TCP. The sparse solutions of $\mathcal{KS}$-TCP can be formulated as
  (P0)
  $$
  \begin{cases}
  \min  \parallel{x}\parallel_{0} \\
  {\rm s.t.} \mathcal{A}x^{m-1}-q\geq0, \quad x\geq0, \quad  \langle{x},\mathcal{A}x^{m-1}-q\rangle=0,
  \end{cases}
$$
where $q$ is a nonnegative vector.
\begin{remark}
When $q$ is a negative vector, $x=(0,0,\cdots,0)^\top$ must satisfy ``$x\geq 0$, $\mathcal{A}x^{m-1}-q\geq 0$, $x^\top(\mathcal{A}x^{m-1}-q)=0$'', that is, the zero vector must be the sparsest solution of {\rm TCP}, which is meaningless to study. Therefore, we set $q$ is a non-negative vector.
\end{remark}

Next, we first explore the properties of the solution set of $\mathcal{KS}$-tensor complementarity problem. In \cite{BHW}, Bai, Huang and Wang showed that the solution set of TCP($q$,$\mathcal{A}$) is nonepmty and compact for any $q\in\mathbb{R}^n$, where $\mathcal{A}$ is a $\mathcal{P}$-tensor. Now, we can get the following theorem.
\begin{theorem}\label{thm0}
$\mathcal{KS}$-{\rm TCP} \eqref{kstcp} has a nonempty and compact solution set for any $q\in\mathbb{R}^n$.
\end{theorem}
\begin{proof}
From the Definition \ref{defks}, we know that a $\mathcal{KS}$-tensor is a $\mathcal{P}$-tensor. Thus $\mathcal{KS}$-TCP \eqref{kstcp} has a nonempty and compact solution set.
\end{proof}
\begin{remark}
In \cite{WMCW}, Wang et al. proved that the $\mathcal{KS}$-tensor equations $\mathcal{A}x^{m-1}=q$ with $q>0$ always has a unique positive solution, that is, there exists $x^{*}>0$ such that $\mathcal{A}({x^{*}})^{m-1}-q=0$. Then $\mathcal{KS}$-{\rm TCP} \eqref{kstcp}
has a solution $x^{*}$, i.e., the solution set of {\rm TCP}($q$,$\mathcal{A}$) is nonepmty when $\mathcal{A}$ is a $\mathcal{KS}$-tensor.
\end{remark}
We want to get the sparse solution of the $\mathcal{KS}$-TCP. A natural question is whether the solution of $\mathcal{KS}$-TCP is a unique solution. If it has a unique solution, then it is meaningless to consider the sparse solution of it. Now, we give an example to show that when $\mathcal{A}$ is a $\mathcal{KS}$-tensor, the solution of $\mathcal{KS}$-TCP is not unique.
\begin{example}\label{noweiyi}
Consider $\mathcal{KS}$-{\rm TCP} \eqref{kstcp}. Let $\mathcal{A}=(a_{i_1i_2i_3i_4})\in \mathbb{R}^{[4,2]}$ with its entries $a_{1111}=1,a_{1112}=-2,a_{1122}=1,a_{2222}=1$ and other $a_{i_1i_2i_3i_4}=0$.
\end{example}
For any vector $x=(x_1,x_2)^{\top}$, we have
\begin{center}
$\mathcal{A}x^3=\begin{pmatrix}x_1^3-2x_1^2x_2+x_1x_2^2\\x_2^3\end{pmatrix}$.
\end{center}
Thus, $x_1(\mathcal{A}x^3)_1=x_1^4-2x_1^3x_2+x_1^2x_2^2$ and $x_2(\mathcal{A})_2=x_2^4$. By the definition of $\mathcal{P}$-tensor, we can see that $\mathcal{A}$ is a $\mathcal{P}$-tensor. And $w_{1111}=1, w_{1112}=-2, w_{1122}=0, w_{2222}=1$. Obviously, $\mathcal{W}$ is a $\mathcal{Z}$-tensor. From
\begin{center}
$\mathcal{W}x^3=\begin{pmatrix}x_1^2(x_1-2x_2)\\x_2^3\end{pmatrix}$,
\end{center}
there must be $x>0$, such that $\mathcal{W}x^3>0$. So $\mathcal{W}$ is a nonsingular $\mathcal{M}$-tensor and $\mathcal{A}$ is a $\mathcal{KS}$-tensor. \\
Taking $q=(0,1)^{\top}$, we have
\begin{center}
$\left\{\begin{array}{l}x_1\geq0\\x_2\geq0\end{array}\right.$, $\left\{\begin{array}{l}x_1^3-2x_1^2x_2+x_1x_2^2\geq0\\x_2^3-1\geq0\end{array}\right.$,
$\left\{\begin{array}{l}x_1(x_1^3-2x_1^2x_2+x_1x_2^2)=0\\x_2(x_2^3-1)=0\end{array}\right.$.
\end{center}
It is easy to see that $x=(0,1)^{\top}$ and $x=(1,1)^{\top}$ are two solutions of this example.

\begin{lemma}\label{lem3.1}{\rm \cite{LQX}}
If $F:  \mathbb{R}^{n}\rightarrow \mathbb{R}^{n}$ is a Z-function, then the following statement holds:\begin{center}
$x\in \mathbb{R}^{n}_{+},y\in \mathbb{R}^{n}_{+},\langle {x,y}\rangle=0\Rightarrow \langle{x,F(y)-F(0)}\rangle\leq0$.
\end{center}
Moreover, if $F(x)=Ax$ is a linear function, then $A$ is a $Z$-matrix, which is equivalent to the following statement:
\begin{equation}\label{3}
  x\in \mathbb{R}_+^n,y\in \mathbb{R}_+^n,\left\langle x,y\right\rangle=0\Rightarrow\left\langle x,Ay\right\rangle\leq0.
\end{equation}
\end{lemma}

\begin{theorem}\label{thm0}
If $\mathcal{A}$ is a $\mathcal{KS}$-tensor which satisfies \eqref{2},
then $F(x)=\mathcal{A}x^{m-1}$ is a $Z$-function.
\end{theorem}
\begin{proof}
According to Proposition \ref{lem1}, in order to prove $F(x)=\mathcal{A}x^{m-1}$ is a Z-function, it suffices to show that for any $x\in \mathbb{R}^{n}_{+}$, $\nabla{F(x)}$ is a Z-matrix. Combining with Lemma \ref{lem3.1}, we only need to show that \eqref{3} holds with $A:=\bigtriangledown{F(x)}$ for any given $x\in \mathbb{R}^{n}_{+}$.

Suppose $y,z\in \mathbb{R}^{n}_{+}$ with $\langle{y},{z}\rangle=0$, easily we can get $y_{i}\geq0,z_{i}\geq0,y_{i}z_{i}=0, \forall{i}\in[n]$. Thus,
\begin{eqnarray*}
     &&\langle{y},\bigtriangledown{F(x)}z\rangle\\
     &=&\langle{y},\bigtriangledown{\mathcal{A}x^{m-1}}z\rangle\\
     &=& \langle{y},(m-1)\mathcal{\bar{A}}x^{m-2}z\rangle \\
   &=&(m-1)\sum_{i=1}^ny_i{\left[\bar{\mathcal{A}}x^{m-2}z\right]}_i\\
   &=&(m-1)\sum_{i=1}^ny_i\lbrack{\left(\bar{\mathcal{A}}x^{m-2}\right)}_{i1}z_1+\cdots+{\left(\bar{\mathcal{A}}x^{m-2}\right)}_{in}z_n\rbrack\\
   &=& (m-1)\sum_{i=1}^ny_i\lbrack\sum_{i_3,\cdots i_{m=1}}^n\bar{a}_{i1i_3\cdots i_m}x_{i_3}\cdots x_{i_m}z_1\\
   &+&\sum_{i_3,\cdots i_{m=1}}^n\bar{a}_{i2i_3\cdots i_m}x_{i_3}\cdots x_{i_m}z_2+\cdots+\sum_{i_3,\cdots i_{m=1}}^n\bar{a}_{ini_3\cdots i_m}x_{i_3}\cdots x_{i_m}z_n\rbrack\\
   &=& (m-1)\sum_{i=1}^n\lbrack\sum_{i_3,\cdots i_{m=1}}^n\bar{a}_{i1i_3\cdots i_m}x_{i_3}\cdots x_{i_m}z_1y_i\\
   &+&\sum_{i_3,\cdots i_{m=1}}^n\bar{a}_{i2i_3\cdots i_m}x_{i_3}\cdots x_{i_m}z_2y_i+\cdots+\sum_{i_3,\cdots i_{m=1}}^n\bar{a}_{ini_3\cdots i_m}x_{i_3}\cdots x_{i_m}z_ny_i\rbrack.\\
   &=& \overset n{\underset{i=1}{\sum y_i}}\sum_{i_{2,}\cdots,i_m=1}^n(a_{ii_2\cdots i_m}+a_{i_2i\cdots i_m}+\cdots+a_{i_2\cdots i_mi})x_{i_2}\cdots x_{i_{m-1}}z_{i_m}\\
   &=&\sum_{i=1}^n[\sum_{i_{2,}\cdots,i_m=1}^n(a_{ii_2\cdots i_m}+a_{i_2i\cdots i_m}+\cdots+a_{i_2\cdots i_mi})x_{i_2}\cdots x_{i_{m-1}}z_{i_m}y_i]\\
   &=&\sum_{i=1}^n[\underset{i_m\neq{i}}{\sum_{i_{2,}\cdots,i_m=1}^n}(a_{ii_2\cdots i_m}+a_{i_2i\cdots i_m}+\cdots+a_{i_2\cdots i_mi})x_{i_2}\cdots x_{i_{m-1}}z_{i_m}y_i\\
   &+&\underset{i_m=i}{\sum_{i_{2,}\cdots,i_m=1}^n}(a_{ii_2\cdots i_m}+a_{i_2i\cdots i_m}+\cdots+a_{i_2\cdots i_mi})x_{i_2}\cdots x_{i_{m-1}}z_{i_m}y_i]\\
   &=&\sum_{i=1}^n\underset{i_m\neq{i}}{\sum_{i_{2,}\cdots,i_m=1}^n}(a_{ii_2\cdots i_m}+a_{i_2i\cdots i_m}+\cdots+a_{i_2\cdots i_mi})x_{i_2}\cdots x_{i_{m-1}}z_{i_m}y_i.\\
\end{eqnarray*}
Furthermore, by Definition \ref{defks}, when $i\neq i_m$, we have
\begin{eqnarray*}
\begin{split}
  &a_{ii_2\cdots i_m}+a_{i_2i\cdots i_m}+\cdots+a_{i_2\cdots i_mi} \\
  =& w_{ii_2\cdots i_m}+n_{ii_2\cdots i_m}+w_{i_2i\cdots i_m}+n_{i_2i\cdots i_m}+\cdots+w_{i_2\cdots i_mi}+n_{i_2\cdots i_mi} \\
   =& w_{ii_2\cdots i_m}+w_{i_2i\cdots i_m}+\cdots+w_{i_2\cdots i_mi}+n_{ii_2\cdots i_m}+n_{i_2i\cdots i_m}+\cdots+n_{i_2\cdots i_mi} \\
   \leq& 0 ,
\end{split}
\end{eqnarray*}
where the last inequality follows from \eqref{2}.
So we have
$\langle{y},\bigtriangledown{F(x)z}\rangle\leq0$. Thus $\bigtriangledown{F(x)}$ is a $Z$-matrix, and $F(x)$ is a $Z$-function.
\end{proof}

\begin{theorem}\label{dengjia}
Let $\mathcal{A}$ be a $\mathcal{KS}$-tensor satisfying \eqref{2} and $q\in \mathbb{R}^{n}_{+}$. The following two systems are equivalent:\\
{\rm (i)}  $x\in \mathbb{R}^{n}_{+}$, $\mathcal{A}x^{m-1}-q \in \mathbb{R}^{n}_{+}$, $\left\langle x,\mathcal{A}x^{m-1}-q\right\rangle=0$;\\
{\rm (ii)}  $x\in \mathbb{R}^{n}_{+}$, $\mathcal{A}x^{m-1}-q=0$.
\end{theorem}
\begin{proof}
Obviously, any solution to system (ii) is a solution to system (i). Next, we prove that any solution to system (i) is a solution to system (ii).\\
Let $x$ be any solution to system (i). Since $\mathcal{A}$ is a $\mathcal{KS}$-tensor satisfying \eqref{2} and $q\in \mathbb{R}^{n}_{+}$, it yields that
\begin{eqnarray}
  0 &\geq& \left\langle \mathcal{A}x^{m-1}-q,\mathcal{A}x^{m-1}\right\rangle  \nonumber\\
   &=& \left\langle \mathcal{A}x^{m-1}-q,\mathcal{A}x^{m-1}-q\right\rangle +\left\langle \mathcal{A}x^{m-1}-q,q\right\rangle\nonumber\\
   &\geq& \left\langle \mathcal{A}x^{m-1}-q,\mathcal{A}x^{m-1}-q\right\rangle \nonumber\\
   &=& \parallel\mathcal{A}x^{m-1}-q\parallel^2_2.\nonumber
\end{eqnarray}
This indicates that $\mathcal{A}x^{m-1}-q=0$, which implies that $x$ is a solution to (ii).
\end{proof}

\begin{remark}\label{re}
From  Theorem \ref{dengjia}, we know that the $(P0)$ problem can be achieved by solving the following problem (P1)
$$
\begin{cases}
\min  e^{\top}x\\
{\rm s.t.}  \mathcal{A}x^{m-1}-q=0,x\geq0,
\end{cases}
$$
where $e=(1,1,\cdots,1)^{\top}$ .
\end{remark}

\section{SQP algorithm for the sparsest solutions of $\mathcal{KS}$-TCP}
In this section, we will apply the SQP algorithm to get the sparse solution of $\mathcal{KS}$-TCP \eqref{kstcp}. The SQP algorithm is an important algorithm for solving optimization problems \cite{M,H1966,sqp43,sqp44,sqp45}.

Consider the specific form of (P1) problem
\begin{eqnarray}\label{yuanwenti}
\begin{cases}
  \min  e^{\top}x, \\
  s.t.  (\mathcal{A}x^{m-1}-q)_i=0, i\in E=\{1,2,\cdots,n\}, \\
    x\geq0.
\end{cases}
\end{eqnarray}
The Lagrange function of \eqref{yuanwenti} is
\begin{eqnarray}
  L(x,\mu,\lambda) &=& e^{\top}x-\sum_{i=1}^n\mu_i(\mathcal{A}x^{m-1}-q)_i-\sum_{i=1}^n\lambda_ix_i \nonumber \\
   &=& e^{\top}x-\lambda^{\top}x-\mu^{\top}(\mathcal{A}x^{m-1}-q),\nonumber
\end{eqnarray}
where $\mu=(\mu_1,\mu_2,\cdots,\mu_n)^{\top}$, $\lambda=(\lambda_1,\lambda_2,\cdots,\lambda_n)^{\top}$ are Lagrange multiplier vector.

After given $(x_k,\mu_k,\lambda_k)$, the constraint function is explicit and the Lagrangian function is approximated by quadratic polynomial. In the $k$th iteration of SQP algorithm, the following forms of quadratic programming subproblems
\begin{eqnarray}\label{ziwenti}
	\begin{cases}
  \min  \frac12d^{\top}B_{k}d+e^{\top}d,\nonumber\\
  s.t.  \mathcal{A}x^{m-1}-q+(m-1)(\mathcal{A}x^{m-2})^{\top}d=0, \\
    x+d\geq0, \nonumber
    \end{cases}
\end{eqnarray}
 are solved, where $B_k$ is a symmetric positive definite matrix. It is an approximation of Hessen matrix of Lagrange function. $d_k$ has a good property that it is the descend direction of many penalty functions, such as, $L_1$ penalty function.

Next, the SQP algorithm is displayed.
\begin{algorithm}
\caption{SQP algorithm for (P1) problem}\label{alg1}
\begin{algorithmic}[1]
\item[{\bf Step 0}] Choose initial point $(x_0,\mu_0,\lambda_0)\in\mathbb{R}^n\times\mathbb{R}^n\times\mathbb{R}^n$, symmetric positive definite matrix $B_0\in\mathbb{R}^{n\times n}$. Denote $h_i(x)=(\mathcal{A}x^{m-1})_i-q_i, i\in E$, $g(x)=x$ and $g(x)_{-}=max\{0,-g_i(x)\}$. Compute
   \begin{center}
    $A_0^E=\bigtriangledown h(x_o)^{\top}=(m-1)[\mathcal{A}x^{m-2}_0]^{\top}$, $A_0=
\begin{pmatrix}A_0^E\\I_n\end{pmatrix}$.
   \end{center}
Choose $\eta\in(0,1/2)$, $\rho\in(0,1)$, $0<\epsilon_1,\epsilon_2\ll{1}$. Set $k:=0$.
\item[{\bf Step 1}] Solve the following subproblem
\begin{eqnarray}\label{sfziwenti}
	\begin{cases}
\min  \frac12d^{\top}B_{k}d+e^{\top}d,\\
s.t.  h(x_k)+A_k^{E}d=0,\\
      x_k+d\geq0,
      \end{cases}
\end{eqnarray}
to get the optimal solution $d_k$.
\item[{\bf Step 2}] If $\|d_k\|_1\leq\epsilon_1$ and $\|h_k\|_1+\|(g_k)_-\|_1\leq\epsilon_2$, stop, an approximate KKT point $(x_k,\mu_k,\lambda_k)$ of the original problem is obtained.
\item[{\bf Step 3}] For a certain merit function, the penalty parameter $\sigma_k$ is chosen so that $d_k$ is the descend direction of the function at $x_k$.
\item[{\bf Step 4}] Let $m_k$ be the smallest nonnegative integer $m$, such that
\begin{center}
$\phi(x_k+\rho^{m}d_k,\sigma_k)-\phi(x_k,\sigma_k)\leq\eta\rho^{m}\bigtriangledown \phi(x_k,\sigma)^{\top}d_k$,
\end{center}
where $\phi$ is $l_1$ merit function and $\bigtriangledown \phi$ is the subgradient of $\phi$. Set $\alpha_k:=\rho^{m_k}$, $x_{k+1}:=x_k+\alpha_{k}d_k$.
\item[{\bf Step 5}] Compute
\begin{center}
$A_{k+1}^E=\bigtriangledown h(x_{k+1})^{\top}$, $A_{k+1}=\begin{pmatrix}A_{k+1}^E\\I_n\end{pmatrix}$
\end{center}
and
\begin{center}
$\begin{pmatrix}\mu_{k+1}\\\lambda_{k+1}\end{pmatrix}=[A_{k+1}A_{k+1}^{\top}]^{-1}A_{k+1}$.
\end{center}
\end{algorithmic}
\end{algorithm}
\begin{algorithm}
\begin{algorithmic}[1]
\item[{}]
\item[{\bf Step 6}] Let
\begin{center}
$s_k=\alpha_kd_k$, $y_k=\bigtriangledown_xL(x_{k+1},\mu_{k+1},\lambda_{k+1})-\bigtriangledown_xL(x_k,\mu_{k+1},\lambda_{k+1})$,
$B_{k+1}=B_k-\frac{B_ks_ks_k^{\top}B_k}{s_k^{\top}B_ks_k}+\frac{z_kz_k^{\top}}
{s_k^{\top}z_k}$,
\end{center}
where $z_k=\theta_ky_k+(1-\theta_k)B_ks_k$, $\theta_k$ is defined as
\begin{center}
$\theta_k=\left\{\begin{array}{lc}1,&s_ky_k\geq0.2s_k^{\top}B_ks_k,\\
	\frac{0.8s_k^{\top}B_ks_k}{s_k^{\top}B_ks_k-s_k^{\top}y_k},&s_k^{\top}y_k<0.2s_k^{\top}B_ks_k.\end{array}\right.$
\end{center}
\item[{\bf Step 7}] Set $k:=k+1$. Go to Step 1.
\end{algorithmic}
\end{algorithm}
\bigskip
\begin{remark}
In order to ensure the global convergence of the SQP method, a merit function is usually used to determine the search step size. For example, the objective function, the penalty function and the augmented Lagrange function can be used as merit function to measure the quality of one-dimensional search. In Step 3 of Algorithm 1, if we select $l_1$ merit function
\begin{center}
$\phi(x,\sigma)=f(x)+\sigma^{-1}[\|h(x)\|_1+\|g(x)_{-}\|_1]$,
\end{center}
where $f(x)=e^{\top}x$, let
\begin{center}
$\tau=\max\{\|\mu_k\|,\|\lambda_k\|\}$.
\end{center}
If $\delta>0$ is selected arbitrarily, the correction rule of defining penalty parameter is
\begin{center}
$\sigma_k=\left\{\begin{array}{lc}\sigma_{k-1},&\sigma_{k-1}^{-1}\geq\tau+\delta\\(\tau+2\delta)^{-1},&\sigma_{k-1}^{-1}<\tau+\delta\end{array}\right.$ .
\end{center}
\end{remark}

Using the KKT condition, \eqref{sfziwenti} is equivalent to
\begin{equation}\label{H1}
  H_1(d,\mu,\lambda)=B_kd-(A_k^E)^{\top}\mu-I_n\lambda+e=0.
\end{equation}
\begin{equation}\label{H2}
  H_2(d,\mu,\lambda)=h(x_k)+A_k^Ed=0.
\end{equation}
\begin{equation}\label{hubu}
  \lambda\geq0, g(x_k)+d\geq0, \lambda^{\top}[g(x_k)+d]=0.\nonumber
\end{equation}
Define
$\Phi(\varepsilon,d,\lambda)=(\phi_1(\varepsilon,d,\lambda),\phi_2(\varepsilon,d,\lambda),\cdots,\phi_m(\varepsilon,d,\lambda))^{\top}$,
where
\begin{center}
$\phi_i(\varepsilon,d,\lambda)=\lambda_i+[g_i(x_k)+d_i-\sqrt{\lambda_i^2+[g_i(x_k)+d_i]^2+2\varepsilon^2}$,
\end{center}

Let $z=(\varepsilon,d,\mu,\lambda)\in \mathbb{R}_+\times \mathbb{R}^n\times \mathbb{R}^n\times \mathbb{R}^n$. Then equations \eqref{H1}-\eqref{H2} are equivalent to
\begin{equation}\label{HZ}
  H(z):=H(\varepsilon,d,\mu,\lambda)=\begin{pmatrix}\varepsilon\\H_1(d,\mu,\lambda)\\H_2(d,\mu,\lambda)\\\Phi(\varepsilon,d,\lambda)\end{pmatrix}=0,\nonumber
\end{equation}
where $v=\bigtriangledown_\varepsilon \Phi(\varepsilon,d,\lambda)=(v_1,\cdots,v_m)^{\top}$, $v_i$ is determined by the following formula
\begin{equation}\label{vi}
  v_i=-\frac{2\varepsilon}{\sqrt{\lambda_i^2+[g_i(x_k)+d_i]^2+2\varepsilon^2}}\nonumber
\end{equation}
and $D_1(z)={\rm diag}(a_1(z),\cdots,a_m(z))$, $D_2(z)={\rm diag}(b_1(z),\cdots,b_m(z))$, where $a_i(z)$ and $b_i(z)$ are determined by the following formulae
$$
 \begin{cases}
  a_i(z)=1-\frac{\lambda_i}{\sqrt{\lambda_i^2+[g_i(x_k)+d_i]^2+2\varepsilon^2}},\\
  b_i(z)=1-\frac{g_i(x_k)+(A_k^I)_id}{\sqrt{\lambda_i^2+[g_i(x_k)+d_i]^2+2\varepsilon^2}}.\\
\end{cases}
$$
Given the parameter $\gamma\in (0,1)$, we define a nonnegative function
\begin{equation}\label{betaz}
  \beta(z)=\gamma\|H(z)\| \min\{1,\|H(z)\|\}.\nonumber
\end{equation}

Now, we give a smooth Newton method for solving the subproblem \eqref{sfziwenti}.
\begin{algorithm}
\caption{}\label{alg1}
\begin{algorithmic}[1]
\item[{\bf Step 0}] Choose $\rho,\eta\in (0,1)$, $\epsilon \geq0$ and $(d_0,\mu_0,\lambda_0)\in \mathbb{R}^n\times\mathbb{R}^n\times\mathbb{R}^n$. Let $z_0=(\varepsilon_0,d_0,\mu_0,\lambda_0)$, $\bar{z}=(\varepsilon_0,0,0,0)$. Choose $\gamma\in(0,1)$, such that $\gamma\mu_0<1$ and $\gamma\|H(z_0)\|<1$. Set $j:=0$.

\item[{\bf Step 1}] If $\|H(z^j)\|=0$, stop. Otherwise, compute $\beta_j=\beta(z_j)$.

\item[{\bf Step 2}] Solve the following equation
\begin{center}
$H(z_j)+H'(z_j)\Delta z_j=\beta_j\bar{z}$
\end{center}
to get $\Delta z_j=(\Delta\epsilon_j,\Delta d_j,\Delta \mu_j,\Delta \lambda_j)$

\item[{\bf Step 3}] Let $m_j$ be the smallest nonnegative integer satisfying
    \begin{center}
    $\|H(z_j+\rho^{m_j}\Delta z_j)\|\leq[1-\sigma(1-\beta\mu_0)\rho^{m_j}]\|H(z_j)\|$.
    \end{center}
Set $\alpha_j:=\rho^{m_j}$, $z_{j+1}=z_j+\alpha_j\Delta z_j$.

\item[{\bf Setp 4}] Set $j:=j+1$. Go to Step 1.
\end{algorithmic}
\end{algorithm}
\bigskip

From the convergence results for SQP algorithm in \cite{M,sqp43} and Remark \ref{re}, we get the following theorem.
\begin{theorem}\label{thm0}
Suppose that \eqref{2} holds and
\begin{center}
$m\|d\|_2^2\leq d^{\top}B_kd\leq M\|d\|_2^2$
\end{center}
holds for constants $m,M>0$ with all $k$ and $d\in \mathbb{R}^n$. If $\|\lambda_k\|_\infty\leq\sigma$ holds for all $k$, then any accumulation point of the point sequence ${x_k}$ generated by Algotighm 1 is the KKT point of (P1) problem, which is also the sparse solution of $\mathcal{KS}$-{\rm TCP} \eqref{kstcp}.
\end{theorem}

\section{Numerical examples}
In this section, we apply Algorithm 1 to solve the sparse solution of $\mathcal{KS}$-TCP. In our tests, the parameters used in Algorithm 1 are chosen as $\rho=0.5$, $\beta=0.1$, $\sigma=0.8$, $\varepsilon_1=10^{-6}$ and $\varepsilon_2=10^{-5}$. Let $B_0$ be an identity matrix. We implement all examples in MATLAB R2019b and run the codes on a PC with a 1.80GHz CPU and 16.00GB of RAM.
\begin{example}\label{a}
Consider $\mathcal{KS}$-{\rm TCP} \eqref{kstcp} with tensor $\mathcal{A}=(a_{i_1i_2i_3i_4})\in \mathbb{R}^{[4,2]}$ and $q=(0,1)^{\top}$, where $\mathcal{A}$ with entries $a_{1111}=1$, $a_{2222}=8$, $a_{1112}=-2$ and other $a_{i_1i_2i_3}=0$.
\end{example}
The numerical results are given in Table 1. We use the command ``rand" in MATLAB to generate the initial point $x_0$, $\lambda_0$, $\mu_0$.  We can get the optimal solution $x^*=(0.0000,0.5000)^{\top}$. In Table \ref{tab:1}, $\mu$ and $\lambda$ are multipliers of equality constraints and inequality constraints respectively, ``Iter" denotes the number of iterations.

\begin{table}[h]
\begin{center}
	\caption{The numerical results for Example \ref{a}}
	\label{tab:1}       
	\begin{tabular}{cccccc}
\hline \hline \noalign{\smallskip}
	 $\mu$ & $\lambda$ & Iter & $x^*$ \\
\hline
 (-2.9303,0.0000)$^{\top}$ & (-0.6207,-0.0005)$^{\top}$ & 17 &(0.0000,0.5000)$^{\top}$\\
(-4.7808,0.0000)$^{\top}$ & (-0.2610,-0.6201)$^{\top}$ & 19 & (0.0000,0.5000)$^{\top}$ \\
 (-3.1071,-0.0000)$^{\top}$ & (-0.3931,-0.0003)$^{\top}$ & 17& (0.0000,0.5000)$^{\top}$  \\
 (-3.8231,0.0000)$^{\top}$& (-0.3192,-0.0004)$^{\top}$ &17 & (0.0000,0.5000)$^{\top}$ \\
  (-3.3756,0.0000)$^{\top}$ & (-0.3400,-0.0001)$^{\top}$ &15 &(0.0000,0.5000)$^{\top}$ \\
\noalign{\smallskip}\hline
	\end{tabular}
\end{center}
\end{table}

\begin{example}\label{b}
Consider $\mathcal{KS}$-{\rm TCP} \eqref{kstcp}, where $\mathcal{A}$ is taken from Example \ref{teshuKSnoZ}. Let $q=(0,1)^{\top}$.
\end{example}
We randomly generate the initial point $x_0$, $\lambda_0$, $\mu_0$ through the command ``rand" in MATLAB. We get the optimal solution $x^*=(0.0000,1.0000)^{\top}$. The detailed numerical results are shown in Table \ref{tab:2}, where $\mu$ and $\lambda$ are multipliers of equality constraints and inequality constraints respectively. ``Iter" denotes the number of iterations.

\begin{table}[h]
\begin{center}
	\caption{The numerical results for Example \ref{b}}
	\label{tab:2}       
	\begin{tabular}{cccccc}
\hline \hline \noalign{\smallskip}
	 $\mu$ & $\lambda$ & Iter & $x^*$ \\
\hline
		(5.2383,0.0000)$^{\top}$ & (-0.2255,-0.0027)$^{\top}$ & 29 &(0.0000,1.0000)$^{\top}$\\
(4.1050,0.0000)$^{\top}$ & (-0.2572,-0.0028)$^{\top}$ & 28 & (0.0000,1.0000)$^{\top}$ \\
 (8.2951,0.0000)$^{\top}$ & (-0.1838,-0.0027)$^{\top}$ & 27& (0.0000,1.0000)$^{\top}$  \\
 (4.2080,0.0000)$^{\top}$& (-0.2535,-0.0028)$^{\top}$ &31 & (0.0000,1.0000)$^{\top}$ \\
  (4.3838,0.0000)$^{\top}$ & (-0.2484,-0.0027)$^{\top}$ &29 &(0.0000,1.0000)$^{\top}$ \\
  \noalign{\smallskip}\hline
	\end{tabular}
\end{center}
\end{table}

\begin{example}\label{c}
Consider $\mathcal{KS}$-{\rm TCP} \eqref{kstcp} with tensor $\mathcal{A}\in\mathbb{R}^{[6,3]}$ and $q=(0,1,1)^{\top}$, where $\mathcal{A}$ with entries $a_{111111}=1=a_{222222}=1=a_{333333}=1$, $a_{123211}=-1,a_{231121}=-2$ and other $a_{i_1i_2i_3i_4i_5i_6}=0$.
\end{example}
Similarly, we use the command ``rand" to randomly generate the initial point $x_0, \lambda_0, \mu_0\in \mathbb{R}^3$.  We show the numerical results in Table \ref{tab:3}.

\begin{table}[h]\footnotesize
\begin{center}
	\caption{The numerical results for Example \ref{c}}
	\label{tab:3}       
	\begin{tabular}{cccccc}
\hline \hline \noalign{\smallskip}
	 $\mu$ & $\lambda$ &  Iter & $x^*$ \\
\hline
	(-3.1351,0.0000,0.0000)$^{\top}$ & (-0.2733,-0.0000,-0.0000)$^{\top}$ & 17 &(0.0000,1.0000,1.0000)$^{\top}$\\
(-3.7267,0.0000,0.0000)$^{\top}$ & (-0.2162,-0.0000,-0.0000)$^{\top}$ & 17 & (0.0000,1.0000,1.0000)$^{\top}$ \\
 (-4.3449,0.0000,0.0000)$^{\top}$ & (-0.2867,-0.0000,-0.0000)$^{\top}$ & 19& (0.0000,1.0000,1.0000)$^{\top}$  \\
 (-2.7719,0.0000,0.0000)$^{\top}$& (-0.2158,-0.0000,-0.0000)$^{\top}$ &17 & (0.0000,1.0000,1.0000)$^{\top}$ \\
  (-2.8052,0.0000,0.0000)$^{\top}$ & (-0.2567,-0.0000,-0.0000)$^{\top}$ &50 &(0.0000,1.0000,1.0000)$^{\top}$ \\
	\noalign{\smallskip}\hline
	\end{tabular}
\end{center}
\end{table}

\begin{example}\label{d}
Consider $\mathcal{KS}$-{\rm TCP} \eqref{kstcp} with tensor $\mathcal{A}\in\mathbb{R}^{[4,4]}$ and $q=(0,1,1,0)^{\top}$, where $\mathcal{A}$ with entries $a_{1111}=a_{2222}=2$, $a_{3333}=a_{4444}=3$, $a_{1432}=-2,a_{3143}=-5$ and other $a_{i_1i_2i_3i_4}=0$.
\end{example}
We perform Algorithm 1 to solve this example with 50 random initial point $x_0, \lambda_0, \mu_0\in \mathbb{R}^4$ uniformly distributed in (0,1). The successive rate is 64\%. We can get the optimal solution $x^*=(0.0000,0.7937,0.6934,0.0000)^{\top}$. We select some numerical results shown in Table \ref{tab:4} with 5 initial points.
\begin{table}[h]\scriptsize
\begin{center}
	\caption{The numerical results for Example \ref{d}}
	\label{tab:4}       
	\begin{tabular}{cccccc}
\hline \hline \noalign{\smallskip}
	 $\mu$ & $\lambda$ &  Iter & $x^*$ \\
\hline
	(0.0000,0.0000,0.0000,1.8544)$^{\top}$ & (-0.5571,-0.0253,-0.0259,0.0060)$^{\top}$ & 82 &(0.0000,0.7937,0.6934,0.0000)$^{\top}$\\
(0.0000,0.0000,0.0000,3.2649)$^{\top}$ & (-0.4879,-0.0000,-0.0000,0.0000)$^{\top}$ & 91 & (0.0000,0.7937,0.6934,0.0000)$^{\top}$ \\
 (0.0000,0.0000,0.0000,1.4169)$^{\top}$ & (-0.2952,0.0000,0.0000,-0.0000)$^{\top}$ & 98& (0.0000,0.7937,0.6934,0.0000)$^{\top}$  \\
 (0.0000,0.0000,0.0000,4.9592)$^{\top}$& (-0.1125,-0.0031,-0.0117,0.0016)$^{\top}$ &80 & (0.0000,0.7937,0.6934,0.0000)$^{\top}$ \\
  (0.0000,0.0000,0.0000,1.6343)$^{\top}$ & (-0.5337,-0.0000,-0.0000,0.0000)$^{\top}$ &79 &(0.0000,0.7937,0.6934,0.0000)$^{\top}$ \\
\noalign{\smallskip}\hline
	\end{tabular}
\end{center}
\end{table}

\begin{example}\label{e}
Consider $\mathcal{KS}$-{\rm TCP} \eqref{kstcp} with tensor $\mathcal{A}\in\mathbb{R}^{[10,9]}$ and $q=(0,0,\cdots,0,1)^{\top}$, where $\mathcal{A}$ with entries $a_{ii\cdots i}=1$, $a_{2677842556}=-3$ and other $a_{i_1i_2\cdots i_{10}}=0$.
\end{example}
Similarly, we use the command ``rand" to randomly generate the initial point $x_0, \lambda_0, \mu_0\in \mathbb{R}^9$. We can get the optimal solution
\begin{center}
$x^*=(0.0000,0.0000,0.0000,0.0000,0.0000,0.0000,0.0000,0.0000,1.0000)^{\top}$.
\end{center}
We show the numerical results in Table \ref{tab:5}.

\begin{table}[h]\scriptsize
\begin{center}
	\caption{The numerical results for Example \ref{e}}\label{tab:5}
	\begin{tabular}{cccccc}
\hline \hline \noalign{\smallskip}
	 $\mu$ & $\lambda$ &  Iter & $x^*$ \\
\hline
	$(0.0000,\cdots,3.8550,0.0000)^{\top}_{9\times1}$ & $(-0.2720,\cdots,-0.1561,0.0000)^{\top}_{9\times1}$ & 122 &$(0.0000,\cdots,0.0000,1.0000)^{\top}_{9\times1}$\\
$(-0.0049,\cdots,0.0219,0.0000)^{\top}_{9\times1}$ & $(-0.3313,\cdots,-0.2868,0.0000)^{\top}_{9\times1}$ & 110 & $(0.0000,\cdots,0.0000,1.0000)^{\top}_{9\times1}$ \\
 $(0.0000,\cdots,0.0000,0.0000)^{\top}_{9\times1}$ & $(0.0052,\cdots,0.0060,-0.4002)^{\top}_{9\times1}$ & 178& $(0.0000,\cdots,0.0000,1.0000)^{\top}_{9\times1}$  \\
 $(0.0055,\cdots,0.0000,0.0000)^{\top}_{9\times1}$& $(-0.3611,\cdots,-0.0939,0.0000)^{\top}_{9\times1}$ &134 & $(0.0000,\cdots,0.0000,1.0000)^{\top}_{9\times1}$ \\
  $(0.0000,\cdots,0.0000,0.0000)^{\top}_{9\times1}$ & $(-0.4377,\cdots,-0.5240,0.0000)^{\top}_{9\times1}$ &113 &$(0.0000,\cdots,0.0000,1.0000)^{\top}_{9\times1}$ \\
\noalign{\smallskip}\hline
	\end{tabular}
\end{center}
\end{table}

From the numerical results in Tables \ref{tab:1}-\ref{tab:5}, we can see that the SPQ algorithm can find the sparse solution of $\mathcal{KS}$-{\rm TCP} problems effectively.

\section{Conclusions}
$\mathcal{KS}$-{\rm TCP} is presented in this paper, we also investigate the sparse solution of it. Based on the properties of  $\mathcal{KS}$-tensor, we show that solving the sparse solution of $\mathcal{KS}$-TCP problem is equivalent to solve the optimization problems with general constraints. We transform the sparse solution of the $\mathcal{KS}$-TCP problem into a polynomial programming problem with linear objective function due to the property of $l_0$ norm. We apply the SQP algorithm to solve it. From the numerical results, it can be seen that SQP algorithm can effectively find the sparse solution of $\mathcal{KS}$-TCP problems.
\\

\end{document}